\documentclass[11pt,a4paper]{amsart}

\usepackage{amssymb}
\usepackage{amsmath}
\usepackage{amsthm}
\usepackage[mathscr]{eucal}
\usepackage{enumerate}

\usepackage{geometry, graphicx, color} 

\usepackage{verbatim} %for comment environment

\usepackage{tikz} %tikz package for drawing

\theoremstyle{plain}
\newtheorem{theorem}{Theorem}[section]

\theoremstyle{plain}
\newtheorem{prop}{Proposition}[section]

\theoremstyle{definition}

\theoremstyle{remark}
\newtheorem{remark}[theorem]{Remark}

\theoremstyle{plain}

\theoremstyle{plain}
\newtheorem{lemma}[theorem]{Lemma}

\newcommand{\del}{\partial}

\def\R{{\mathbb{R}}}

\def\FT{{\mathcal{F}}}

\def\sN{{\mathcal{N}}}
\def\sS{{\mathcal{S}}}

\def\supp{{\text{\upshape supp}\;}}

\newcommand{\embed}{\hookrightarrow}

\newcommand{\ts}{\textstyle}

\newcommand{\<}{\langle}
\renewcommand{\>}{\rangle}

\renewcommand{\^}[1]{\widehat{#1}}

\renewcommand{\~}[1]{\widetilde{#1}}

\begin{document}

\title[Almost critical LWP for 2D NLW with null forms]{Almost critical well-posedness for  nonlinear\\ wave equation with $Q_{\mu\nu}$ null forms in 2D}


\author[Grigoryan]{Viktor Grigoryan$^1$}
\address{$^1$  Department of Mathematics \\ Occidental College \\  1600 Campus Road, Los Angeles, California 90041}
\email{vgrigoryan@oxy.edu}
\thanks{}

\author[Nahmod]{Andrea R. Nahmod$^2$}
\address{$^2$  
Department of Mathematics \\ University of Massachusetts\\  710 N. Pleasant Street, Amherst MA 01003}
\email{nahmod@math.umass.edu}
\thanks{$^2$ The second author is funded in part by NSF DMS 1201443. She also thanks the Radcliffe Institute for Advanced Study at Harvard U. for its hospitality during the final stages of this work.}

%\date{}
%\subjclass{35L70}

\begin{abstract}
In this paper we prove an optimal local well-posedness result for the 1+2 dimensional system of nonlinear wave equations (NLW) with quadratic null-form derivative nonlinearities $Q_{\mu\nu}$. The Cauchy problem for these equations is known to be ill-possed for data in the Sobolev space $H^s$ with $s<5/4$ for all the basic null-forms, except $Q_0$. However, the scaling analysis predicts local well-posedness all the way to the critical regularity of $s_c=1$. Following Gr\"{u}nrock's result for the quadratic derivative NLW, we consider initial data in the Fourier-Lebesgue spaces $\^{H}_s^r$, which coincide with the Sobolev spaces of the same regularity for $r=2$, but scale like lower regularity Sobolev spaces for $1<r<2$. Here we obtain local well-posedness for the range $s>1+\frac{1}{r}$, $1<r\leq 2$, which at one extreme coincides with $H^{\frac{3}{2}+}$ Sobolev space result, while at the other extreme establishes local well-posedness for the model null-form problem for the almost critical Fourier-Lebesgue space $\^{H}_{2+}^{1+}$. Using appropriate multiplicative properties of the solution spaces and relying on bilinear estimates for the $Q_{\mu\nu}$ forms, we prove almost critical local well-posedness for the Ward wave map problem as well.
\end{abstract}

\maketitle

%\tableofcontents

%\newpage

%%%%%%%%%%%%%%%%%%%%%%%%%%%%%%%%%%%%%%%%%%%%%%%%%%%%%%%%%%%%%%%%%%%%%%%%%%%%%%%%%%%%%%%%%%%%%%%%%%%%%%%%%%%%%%%%%%
%%%%%%%%%%%%%%%%%%%%%%%%%%%%%%%%%%%%%%%%%%%%%%%%%%%%%%%%%%%%%%%%%%%%%%%%%%%%%%%%%%%%%%%%%%%%%%%%%%%%%%%%%%%%%%%%%%
%%%%%%%%%%%%%%%%%%%%%%%%%%%%%%%%%%%%%%%%%%%%%%%%%%%%%%%%%%%%%%%%%%%%%%%%%%%%%%%%%%%%%%%%%%%%%%%%%%%%%%%%%%%%%%%%%%

\section{Introduction}
Consider systems of nonlinear wave equations (NLW),
\begin{equation}\label{Q}
 \Box u^I=Q(u^J, u^K),
\end{equation}
where $(u^I):R^{1+2}\to \R^m$, and $Q$ is a bilinear form inhibiting a null structure. That is, $Q$ can be written as a linear combination of the three basic null forms of Klainerman \cite{klainerman:null1} (see also \cite{klainerman:null2, christ:null}).
\begin{equation}\label{null_forms}\begin{split}
 Q_0(f, g) & =\del_t f\del_t g-\nabla f\cdot \nabla g,\\
 Q_{ij}(f, g) & =\del_i f\del_j g-\del_j f \del_i g,\\
 Q_{0j}(f, g) & =\del_t f\del_j g-\del_j f \del_t g.
\end{split}
\end{equation}
Here $\del_j$ stands for spatial derivatives, and $\nabla$ is the spatial gradient.

We are interested in the local well-posedness question for the system \eqref{Q} for initial data in Fourier-Lebesgue spaces $\^{H}_s^r$. More precisely, we consider the Cauchy problem for \eqref{Q} with initial conditions
\begin{equation}\label{Q_IC}
(u^I, \del_t u^I)|_{t=0}=(f^I, g^I)\in \^{H}_s^r\times \^{H}_{s-1}^r,
\end{equation}
and wish to establish local well-posedness for a range of the exponents $(r, s)$, which achieves almost critical well-posedness steaming from scaling considerations. The Fourier-Lebesgue spaces $\^{H}_s^r$ have been previously successfully used to achieve improved regularity results for a variety of equations (see e.g. \cite{hormander}, \cite{vargas-vega}, \cite{grunrock:mkdv}, \cite{grunrock:nls}, \cite{christ}, \cite{grunrock-herr}, \cite{grunrock:wave}).

Although systems \eqref{Q} were studied as standalone problems before, nonlinear terms with null structure also naturally arise in physical and geometric problems. One such example, which we will study in this paper in detail, is the Ward wave map system. It was introduced by Ward in \cite{ward:soliton} as a two dimensional completely integrable system, with its linear part invariant under Lorentz transformations. The Ward system can be realized as a dimensional reduction of the anti-self-dual Yang Mills equation (ASDYM) with split signature in $\R^{2+2}$, and can also be obtained from the space-time Monopole equation via gauge fixing (see e.g. \cite{dai-terng-uhlenbeck} for more details). The Ward wave map equation has the form
\begin{equation}\label{ward}
 (J^{-1}J_t)_t-(J^{-1}J_x)_x-(J^{-1}J_y)_y-[J^{-1}J_t, J^{-1}J_y]=0,
\end{equation}
where
\[
 J:\R^{1+2}\to U(n)
\]
is a $U(n)$ (or $SU(n)$) valued function, and hence $J^{-1}=J^*$, while $(x, y)\in\R^2$, and $[\cdot, \cdot]$ is the Lie bracket on $U(n)$. Using the product rule, \eqref{ward} can be written as
\[
 J^{-1}J_{tt}-J^{-1}\Delta J+J^{-1}_tJ_t-(\nabla J^{-1})\nabla J-J^{-1}J_tJ^{-1}J_y+J^{-1}J_yJ^{-1}J_t=0.
\]
Multiplying the last equation by $J$ on the left and using $(\del J)J^{-1}=-J\del J^{-1}$, the Ward wave map equation will become
\begin{equation}\label{ward_null}
 \Box J+JQ_0(J^{-1},J)+JQ_{02}(J^{-1}, J)=0.
\end{equation}
Notice that the nonlinearity in the last equation  is cubic in the unknown $J$, and has a null structure in the terms appearing with derivatives. 

System \eqref{Q} is a particular example of the more general quadratic derivative NLW,
\begin{equation}\label{dudu}
 \Box u=\del u \del u,
\end{equation}
where $\del u$ is the space-time gradient of $u$. Equation \eqref{dudu} is invariant under the scaling
\begin{equation}\label{scaling}
 (t, x)\mapsto (\lambda t, \lambda x).
\end{equation}
That is, if $u$ is a solution to \eqref{dudu} in $\R^{1+n}$, then so is $u_\lambda (t, x)=u(\lambda t, \lambda x)$. Under this scaling, the homogeneous Sobolev norm of the initial data scales as
\[
 \|u_\lambda(0, \cdot)\|_{\dot{H}^s(\R^n)}=\lambda^{s-\frac{n}{2}}\|u(0, \cdot)\|_{\dot{H}^s(\R^n)},
\]
and $s_c=\frac{n}{2}$ is called the critical exponent, as the $\dot{H}^{s_c}$ norm of the initial data is preserved under the scaling. Under general scaling considerations, one expects a local well-posedness for data in the Sobolev space $H^s$ for $s>s_c$ (subcritical regime), global existence for small data in $\dot{H}^{s_c}$ (critical regime), and some form of ill-posedness for data in $H^s$ for $s<s_c$ (supercritical regime). 

As the critical exponent in $\R^{1+n}$ is $s_c=\frac{n}{2}$, it is expected that the local well-posedness must hold for data in the Sobolev space $H^s$, $s>\frac{n}{2}$. This can be proved in dimensions $n\geq 4$ with Strichartz estimates approach, however, it is known to be false in dimensions $n=2, 3$.

In dimension $n=3$, the almost critical local well-posedness for null-form quadratic derivative NLW was proved by Klainerman and Machedon \cite{kl-mac:null}, while for the general quadratic derivative NLW \eqref{dudu}, the local well-posedness for $s>2$ was proved by Ponce and Sideris \cite{ponce-sideris}, which is sharp in light of the counterexamples of Linblad \cite{lindblad:1}, \cite{lindblad:2}, \cite{lindblad:3}.

In dimension $n=2$ the almost critical LWP  for the $Q_0$ null-form NLW was showed by Klainerman and Selberg \cite{kl-sel} in the context of wave maps, but it is known to be false for the other null forms, for which the best result is for data in $H^s$ with $s>\frac{5}{4}$ by Zhou \cite{zhou}, who also showed that it is sharp. We also observe that the sharp bilinear $X^{s, b}$ estimates of Foschi-Klainerman \cite{foschi-klainerman:bilinear} for the solutions of the free wave equation associated with the Cauchy problem with data in $H^s$ are also $\frac{1}{4}$ derivative above the scaling regularity. We will see that the $\^{H}_s^r$ space approach with $1<r<2$ circumvents the counterexamples of Zhou and Foschi-Klainerman. 

The best result for the general quadratic derivative NLW is for $s>\frac{7}{4}$, which can be shown by the Strichartz estimates approach.

The null structure in the Ward equation \eqref{ward_null} was exploited by Czubak in \cite{magda:thesis} to prove local well-posedness for the Ward wave map problem for data in $H^s(\R^2)$ for $s>\frac{5}{4}$. In light of Zhou's results, this result is also a $\frac{1}{4}$ derivative above the scaling prediction, since the Ward equation scales as \eqref{dudu}.

Recently, Gr\"{u}nrock showed in \cite{grunrock:wave} that the gap to the almost criticality for the general quadratic NLW \eqref{dudu} can be closed in dimension $n=3$ by considering initial data in the Fourier-Lebesgue spaces $\^{H}_s^r$. These space are defined by their norms as
\[
 \|f\|_{\^{H}_s^r}=\|\<\xi\>^s\^{f}\|_{L_\xi^{r'}}, \qquad \frac{1}{r}+\frac{1}{r'}=1,
\]
where $\^{f}$ stands for the Fourier transform of $f$, and\footnote{Alternatively, one can use the Japanese bracket $\<\xi\>=\sqrt{1+|\xi|^2}\simeq \sqrt{1+|\xi|^2}$.} $\<\xi\>=1+|\xi|$. The norm of the corresponding homogeneous space is $\|f\|_{\dot{\^{H}}_s^r}=\||\xi|^s\^{f}\|_{L_\xi^{r'}}$. Under the scaling \eqref{scaling}, the norm of the initial data in the homogeneous Fourier-Lebesgue spaces scales as
\[
  \|u_\lambda(0, \cdot)\|_{\dot{\^{H}}_s^r(\R^n)}=\lambda^{s-\frac{n}{r}}\|u(0, \cdot)\|_{\dot{\^{H}}_s^r(\R^n)},
\]
so the critical exponent for these spaces is $s_c^r=\frac{n}{r}$. Comparing how the homogeneous Sobolev and Fourier-Lebesgue norms scale, we observe the following correspondence in terms of scaling,
\[
 \dot{\^{H}}_s^r\sim \dot{H}^\sigma, \qquad \text{if} \qquad \sigma=s+n\left (\frac{1}{2}-\frac{1}{r}\right ).
\]
Gr\"{u}nrock established local well-posedness for data in the space $\^{H}_s^r$ for $s>\frac{2}{r}+1$, $1<r\leq 2$. Thus, his range of exponents almost reaches criticality at the endpoint $r=1$, since for this $r$, $s>\frac{2}{r}+1=\frac{3}{r}$, which gives the critical exponent $(s, r)=(3, 1)$ in dimension $n=3$. His approach relies on free wave interaction estimates of Foschi-Klainerman \cite{foschi-klainerman:bilinear} that come with a factor of $||\tau|-|\xi||^{\frac{n-3}{2}}$, which becomes unbounded near  the null cone $|\tau|=|\xi|$ in two dimensions. Thus, Gr\"{u}nrock's result cannot be directly generalized to a LWP result for the general quadratic NLW \eqref{dudu} in dimension $n=2$. However, if the nonlinearity has enough cancellation along the null cone to offset this factor, then the arguments can be salvaged, leading to a LWP for these special nonlinearities for a range of exponents $(s, r)$ that reach almost criticality. We follow this approach to prove the LWP for the Cauchy problem \eqref{Q}-\eqref{Q_IC} for all the null forms \eqref{null_forms}.

The main results of this paper are the following two theorems.

\begin{theorem}[LWP for null-form NLW]\label{main_Q}
 Let $1<r\leq 2$, $s>\frac{1}{r}+1$, then the Cauchy problem \eqref{Q}-\eqref{Q_IC} is locally well-posed for data in the space $\^{H}_s^r\times \^{H}_{s-1}^r$. 
\end{theorem}

We also consider the Cauchy problem for the Ward equation \eqref{ward_null} with data in the Fourier Lebesgue spaces and will prove the following result.
\begin{theorem}[LWP for Ward]\label{ward_LWP}
 Let $1<r\leq 2$, $s>\frac{1}{r}+1$, then the Cauchy problem for the Ward equation \eqref{ward_null} is locally well-posed for data in the space
\[
(J, \del_tJ)\big{|}_{t=0}\in \^{H}_s^r\times \^{H}_{s-1}^r.
\] 
\end{theorem}

\begin{remark}
 Notice that at one extreme, $(s, r)=(\frac{3}{2}, 2)$, this result coincides with the local well-posedness for data in $H^{\frac{3}{2}+}$, which is above the best know result of $H^{\frac{5}{4}+}$, while at the other extreme, $(s, r)=(2, 1)$, we obtain the almost critical local well-posedness in the space $\^{H}_{2+}^{1+}$.

 The region for the $(s, \frac{1}{r})$ exponents, for which the local well-posedness holds is shaded in Figure 1. Notice that the bottom and right edges of the region are not included. The dotted line segment in Figure 1 connects the sharp Sobolev results $H^{\frac{5}{4}+}$ with the critical $\^{H}_2^1$. For the region above this dotted line, one can prove the needed bilinear estimates for solutions of the free wave equation, however, below the shaded region, the estimates require placing the nonlinearity in the $X_{\sigma, b'}^r$ spaces for $b'<0$. These estimates cannot be transfered to estimates for general $X_{s, b}^r$ functions with the transfer principle, Proposition \ref{transfer_prop}. It would be interesting to see whether inhomogeneous estimates of D'Ancona, Foschi and Selberg \cite{DFS_2d} can be generalized to this case, yielding improved local well-posedness for the exponents $(s, r)$, however, we do not pursue this matter here, since we are able to almost reach critical regularity with our approach.
\end{remark}

\begin{figure}[h]
\begin{tikzpicture}[domain=0.5:2, x=5cm, y=2.5cm]
%\draw[very thin,color=gray] (-0.1,-0.1) grid (1.2,3.1);
\draw[->] (-0.1,0) -- (1.6,0) node[below] {$\frac{1}{r}$};
\draw[->] (0,-0.1) -- (0,2.8) node[left] {$s$};
\shade[top color=gray,bottom color=gray!50] 
      (0.5,1.5) -- (0.5, 2.7) -- (1, 2.8) -- (1, 2) -- cycle;
\draw[dashed, domain=0.5:1] plot (\x,\x+1);
\draw[color=blue, domain=0.5:1] plot (\x,2*\x);

\draw (0.75, 1.75) node[above=5pt] {$s=\frac{1}{r}+1$};

\draw (0.75, 1.5) node[right] {$s=\frac{2}{r}$};

\draw plot[only marks,mark=ball, mark options={color=black}] coordinates{(0.5,1.5)};
\draw[color=black] plot[only marks,mark=ball] coordinates{(0.5,1.25) (0.5,1) (1,2)};

\draw[color=gray, dashed] (0.5,1.5) -- (0.5,0) node[below] {$\frac{1}{2}$};
\draw[color=gray, dashed] (1,2.8) -- (1,0)  node[below] {$1$};
\draw[color=gray, dashed] (0.5,1) -- (0,1)  node[left] {$1$};
\draw[color=gray, dashed] (0.5,1.5) -- (0,1.5)  node[left] {$\frac{3}{2}$};
\draw[color=gray, dashed] (0,1.25)  node[left] {$\frac{5}{4}$};

\draw[color=black, thick, dotted] (1, 2) -- (0.5, 1.25);

%\shade[top color=gray,bottom color=gray!50] 
%      (0.5,1.5) -- (1, 2) -- (0.5, 1.25) -- cycle;
\end{tikzpicture}
\caption{The shaded region represents the range of indeces for which LWP for data in space $\^{H}_s^r\times \^{H}_{s-1}^r$ holds.}
\end{figure}
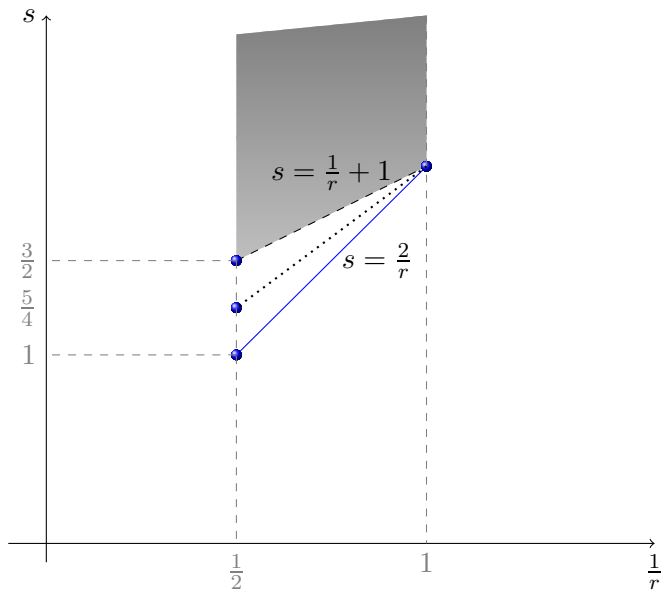

As we mentioned above, the approach via the transfer principle and free wave estimates is not well-suited for the general equation \eqref{dudu}. For this equation, an upcoming joint result of the first author with A. Tanguay \cite{grig-tang} improves the well-posedness range from the best known Sobolev result, which corresponds to a $\frac{1}{12}$ derivative improvement in the Sobolev exponent. Their approach uses generalizations of bilinear estimates in the inhomogeneous norms of D'Ancona, Foschi and Selberg \cite{DFS_2d} to the Fourier-$L^{r'}$ based spaces.

The rest of the paper is organized as follows. in Section \ref{sec2} the solution spaces are introduced, and Theorem \ref{main_Q} is reduced to a bilinear estimate for solutions of the free wave equation. This is done via the general LWP theorem, which is stated in Appendix \ref{appB}, and the transfer principle, which is proved in Appendix \ref{appA}. In Section \ref{sec3} we prove the main bilinear estimate. In section \ref{sec4} using appropriate multiplicative estimates of the solutions spaces, we show that the bilinear estimates of Section \ref{sec3} imply trilinear estimates for $wQ(u, v)$, leading to the almost critical LWP of the Ward wave map problem.

%%%%%%%%%%%%%%%%%%%%%%%%%%%%%%%%%%%%%%%%%%%%%%%%%%%%%%%%%%%%%%%%%%%%%%%%%%%%%%%%%%%%%%%%%%%%%%%%%%%%%%%%%%%%%%%%%%
%%%%%%%%%%%%%%%%%%%%%%%%%%%%%%%%%%%%%%%%%%%%%%%%%%%%%%%%%%%%%%%%%%%%%%%%%%%%%%%%%%%%%%%%%%%%%%%%%%%%%%%%%%%%%%%%%%
%%%%%%%%%%%%%%%%%%%%%%%%%%%%%%%%%%%%%%%%%%%%%%%%%%%%%%%%%%%%%%%%%%%%%%%%%%%%%%%%%%%%%%%%%%%%%%%%%%%%%%%%%%%%%%%%%%

\section{The $X_{s, b}^r$ space, reduction to a bilinear estimate}\label{sec2}
The local in time solution is obtained via a contraction principle in a suitable solution space. For this we will use time restriction spaces based on the $X_{s, b}^r$ space, which is a Fourier-$L^{r'}$ analogue of the wave-Sobolev space $X^{s, b}$ and is defined by the norm
\[
 \|u\|_{X_{s, b}^r}=\|\<\xi\>\<|\tau|-|\xi|\>\~{u}\|_{L_{\tau, \xi}^{r'}},
\]
where $\~{u}$ denotes the time-space Fourier transform of $u$. The time restriction space is then
\[
 X_{s, b; T}^r=\left \{u=U|_{[-T, T]\times \R^n} : U\in X_{s, b}^r\right \},
\]
with its norm defined as
\[
 \|u\|_{X_{s, b; T}^r}=\inf \left \{\|U\|_{X_{s, b}^r} : U|_{[-T, T]\times \R^n}=u \right \}.
\]
For us, $b>\frac{1}{r}$, which guarantees the embeddings (see Proposition \ref{transfer_prop})
\[
 X_{s, b}^r\subset C(\R, \^{H}_s^r), \qquad \text{and} \qquad X_{s, b; T}^r\subset C([-T, T], \^{H}_s^r).
\]

As the wave operator is of second order in time, we also need to separately estimate the time derivative of the solution. Hence, we define our solution space, $Z_{s, b}^r$, by the norm
\[
 \|u\|_{Z_{s, b}^r}=\|u\|_{X_{s, b}^r}+\|\del_t u\|_{X_{s-1, b}^r}.
\]
The time restriction space $Z_{s, b; T}^r$ and its norm are defined as before.

In the sequel we also make use of the notation $\^{L}^r_{t, x}=X_{0, 0}^r$, and $\^{L}^r=\^{H}_0^r$, where the last norm can be taken either with respect to the time or space variables.

By the general well-posedness, Theorem \ref{gen_LWP}, Theorem \ref{main_Q} will follow from the following two estimates
\begin{equation}\label{null}
 \|Q(u, v)\|_{X_{s-1, b+\epsilon-1}^r} \lesssim \|u\|_{X_{s, b}^r}\|v\|_{X_{s, b}^r}, 
\end{equation}
and
\begin{align*}
 \|Q(u, v)-Q(U, V)\|_{X_{s-1, b+\epsilon-1}^r} & \lesssim \left (\|u\|_{X_{s, b}^r}+\|v\|_{X_{s, b}^r}+\|U\|_{X_{s, b}^r}+\|V\|_{X_{s, b}^r}\right )\\
 &\qquad\qquad\qquad \times\left (\|u-U\|_{X_{s, b}^r}+\|v-V\|_{X_{s, b}^r}\right ). \label{null_Lip}
\end{align*}
The second estimate will trivially follow from the first one, as $Q$ is bilinear and, hence,
\[
 Q(u, v)-Q(U, V)=Q(u-U, v)+Q(U, v-V).
\]
Thus, the proof of Theorem \ref{main_Q} reduces to proving estimate \eqref{null} for all the null forms \eqref{null_forms}.

Since $b+\epsilon-1<0$, estimate \eqref{null} will follow from
\begin{equation}\label{null_0}
 \|Q(u, v)\|_{X_{s-1, 0}^r} \lesssim \|u\|_{X_{s, b}^r}\|v\|_{X_{s, b}^r}.
\end{equation}
We denote the symbol of $Q$ by $m$, that is,
\[
 \~{Q(u, v)}(\tau, \xi)=\int\limits_{\alpha+\beta=\tau}\int\limits_{\eta+\zeta=\xi} m(\alpha, \eta; \beta, \zeta) \~{u}(\alpha, \eta)\~{v}(\beta, \zeta)\, d\alpha d\beta d\eta d\zeta.
\]
Denoting the corresponding bilinear operator with the space-normalized symbol,
\[
 \frac{m(\alpha, \eta; \beta, \zeta)}{|\eta||\zeta|},
\]
by $q$, we will have
\[
 |\FT_{t, x}Q(u, v)|\lesssim |\FT_{t, x}q(\del_x u, \del_x v)|,
\]
where $\FT_{t, x}$ is the space-time Fourier transform operator. Then estimate \eqref{null_0} will follow from
\begin{equation}\label{q_0}
 \|\Lambda^{\sigma}q(u, v)\|_{\^{L}^r_{t, x}} \lesssim \|u\|_{X_{\sigma, b}^r}\|v\|_{X_{\sigma, b}^r},
\end{equation}
where $\sigma=s-1$, and the operator $\Lambda^\sigma$ has the multiplier $\<\xi\>^\sigma$, that is,
\[
\^{\Lambda^\sigma \phi}(\xi)=\<\xi\>^\sigma\^{\phi}(\xi).
\]

By the transfer principle\footnote{Here $\Lambda^\sigma$ in front of the bilinear form will be taken care of via a simple Leibniz type estimate on the Fourier side, so its presence does not effect applicability of the transfer principle.}, Proposition \ref{transfer_prop}, estimate \eqref{q_0} will follow from the following bilinear estimate for free waves
\begin{equation}\label{bi_q_free}
  \|\Lambda^\sigma q(u_\pm, v_{[\pm ]})\|_{\^{L}^r_{t, x}}\lesssim \|u_0\|_{\^{H}_\sigma^r}\|v_0\|_{\^{H}_\sigma^r},
\end{equation}
where $[\pm]$ denotes an independent choice of signs, and
\begin{equation*}
 u_\pm(t)=e^{\pm itD}u_0, \qquad v_\pm(t)=e^{\pm itD}v_0.
\end{equation*}

The next section is dedicated to proving \eqref{bi_q_free} in which we follow the method of Gr\"{u}nrock \cite{grunrock:wave}, while relying on calculations of Foschi-Klainerman \cite{foschi-klainerman:bilinear}.

%%%%%%%%%%%%%%%%%%%%%%%%%%%%%%%%%%%%%%%%%%%%%%%%%%%%%%%%%%%%%%%%%%%%%%%%%%%%%%%%%%%%%%%%%%%%%%%%%%%%%%%%%%%%%%%%%%
%%%%%%%%%%%%%%%%%%%%%%%%%%%%%%%%%%%%%%%%%%%%%%%%%%%%%%%%%%%%%%%%%%%%%%%%%%%%%%%%%%%%%%%%%%%%%%%%%%%%%%%%%%%%%%%%%%
%%%%%%%%%%%%%%%%%%%%%%%%%%%%%%%%%%%%%%%%%%%%%%%%%%%%%%%%%%%%%%%%%%%%%%%%%%%%%%%%%%%%%%%%%%%%%%%%%%%%%%%%%%%%%%%%%%

\section{Proof of the main bilinear estimate}\label{sec3}
We first observe that by symmetry we only need to consider the $(++)$ and $(+-)$ cases in \eqref{bi_q_free}. Defining $P_{\pm}(\eta)=|\xi-\eta|\pm |\eta|$ with $\nabla P_\pm(\eta)=\frac{\eta-\xi}{|\eta-\xi|}\pm\frac{\eta}{|\eta|}$, and using the properties of the $\delta$-distribution, we have (see \cite[Secions 3, 4]{foschi-klainerman:bilinear})
\begin{equation*}
 \FT u_+v_\pm (\xi, \tau)\simeq \int_{P_\pm(\eta)=\tau}\frac{dS_\eta}{|\nabla P_\pm (\eta)|}\^{u_0}(\eta)\^{v_0}(\xi-\eta).
\end{equation*}
The set $\{P_+(\eta)=\tau\}$ is an ellipsoid of rotation, while the set $\{P_-(\eta)=\tau\}$ is a hyperboloid of rotation, so we refer to the $(++)$ and $(+-)$ cases as elliptic and hyperbolic respectively.

We will prove estimate \eqref{bi_q_free} for the form $q_{12}$ corresponding to the form $Q_{12}$. The proofs for $q_0$ and $q_{0j}$ are similar, and will be outlined in remarks. Before proceeding, we observe the following bounds for the symbol of the bilinear operator $q_{12}$ (see \cite[Lemma 13.2]{foschi-klainerman:bilinear})
\begin{align}
 \frac{\eta\wedge\zeta}{|\eta||\zeta|} & =\frac{\eta\wedge(\xi-\eta)}{|\eta||\xi-\eta|}\lesssim \frac{|\xi|^{\frac{1}{2}}(|\eta|+|\xi-\eta|-\xi|)^{\frac{1}{2}}}{|\eta|^\frac{1}{2}|\xi-\eta|^{\frac{1}{2}}} \label{q_bound_el},\\
 \frac{\eta\wedge\zeta}{|\eta||\zeta|} & =\frac{\eta\wedge(\xi-\eta)}{|\eta||\xi-\eta|}\lesssim \frac{|\xi|^{\frac{1}{2}}(|\xi|-||\eta|-|\xi-\eta||)^{\frac{1}{2}}}{|\eta|^\frac{1}{2}|\xi-\eta|^{\frac{1}{2}}} \label{q_bound_hyp}.
\end{align}
The first estimate above will be useful in the elliptic case, while the second estimate will be useful in the hyperbolic case.

%%%%%%%%%%%%%%%%%%%%%%%%%%%%%%%%%%%%%%%%%%%%%%%%%%%%%%%%%%%%%%%%%%%%%%%%%%%%%%%%%%%%%%%%%%%%%%%%%%%%%%%%%%%%%%%%%%
%%%%%%%%%%%%%%%%%%%%%%%%%%%%%%%%%%%%%%%%%%%%%%%%%%%%%%%%%%%%%%%%%%%%%%%%%%%%%%%%%%%%%%%%%%%%%%%%%%%%%%%%%%%%%%%%%%
%%%%%%%%%%%%%%%%%%%%%%%%%%%%%%%%%%%%%%%%%%%%%%%%%%%%%%%%%%%%%%%%%%%%%%%%%%%%%%%%%%%%%%%%%%%%%%%%%%%%%%%%%%%%%%%%%%

\subsection{The elliptic case}
We choose $0<s_{1, 2}<\frac{1}{r}$ with $s_1+s_2=\frac{1}{r}$, and use H\"{o}lder's inequality to get
\[
\begin{split}
 & |\FT q_{12}(u_+, v_+)| \lesssim \left (\int_{P_+(\eta)=\tau} \frac{dS_\eta}{|\nabla P_\pm (\eta)|} \left |\frac{\eta\wedge(\xi-\eta)}{|\eta||\xi-\eta|}\right|^r |\eta|^{-s_1 r}|\xi-\eta|^{-s_2 r}\right )^\frac{1}{r}\\
 & \qquad\qquad\qquad\qquad\qquad\qquad  \times \left (\int_{P_+(\eta)=\tau} \frac{dS_\eta}{|\nabla P_+ (\eta)|} \left |\^{\Lambda^{s_1} u_0}(\eta)\^{\Lambda^{s_2} v_0}(\xi-\eta)\right |^{r'}\right )^\frac{1}{r'}.
\end{split}
\]
Using the bound \eqref{q_bound_el}, we have for the first factor above
\begin{align*}
 & I:=\int_{P_+(\eta)=\tau} \frac{dS_\eta}{|\nabla P_\pm (\eta)|} \left |\frac{\eta\wedge(\xi-\eta)}{|\eta||\xi-\eta|}\right|^r |\xi-\eta|^{-s_1 r}|\eta|^{-s_2 r}\\
 & \qquad \lesssim \int_{P_+(\eta)=\tau} \frac{dS_\eta}{|\nabla P_\pm (\eta)|} |\xi|^{\frac{r}{2}}|\eta|^{-(s_1+\frac{1}{2})r}|\xi-\eta|^{-(s_2+\frac{1}{2})r}(|\eta|+|\xi-\eta|-|\xi|)^{\frac{r}{2}}\\
 & \qquad\qquad =|\tau-|\xi||^{\frac{r}{2}}|\xi|^\frac{r}{2}\int \frac{\delta(\tau-|\eta|-|\xi-\eta|)}{|\eta|^{(s_1+\frac{1}{2})r}|\xi-\eta|^{(s_2+\frac{1}{2})r}}\, d\eta.
\end{align*}
We will only need estimates for the cases $s_1=0, s_2=\frac{1}{r}$ and $s_2=0, s_1=\frac{1}{r}$, for both of which
\[
 \ts\max\{(s_1+\frac{1}{2})r, (s_2+\frac{1}{2})r\}=1+\frac{r}{2} > \frac{3}{2}.
\]
Now using \cite[Proposition 4.3]{foschi-klainerman:bilinear}, we get
\begin{equation*}
 I\lesssim |\tau-|\xi||^{\frac{r}{2}}|\xi|^\frac{r}{2} \tau^A (\tau-|\xi|)^B,
\end{equation*}
where 
\begin{align*}
 \ts A & \ts =\max \{(s_1+\frac{1}{2})r, (s_2+\frac{1}{2})r, \frac{3}{2}\}-(s_1+\frac{1}{2})r-(s_2+\frac{1}{2})r=-\frac{r}{2},\\
 \ts B & \ts =1-\max \{(s_1+\frac{1}{2})r, (s_2+\frac{1}{2})r, \frac{3}{2}\}=-\frac{r}{2}.
\end{align*}
Thus,
\begin{equation*}
 I\lesssim \frac{|\tau-|\xi||^{\frac{r}{2}}|\xi|^\frac{r}{2}}{|\tau-|\xi||^{\frac{r}{2}}\tau^{\frac{r}{2}}}\lesssim 1,
\end{equation*}
since $\tau=|\xi-\eta|+|\eta|\geq |\xi|$.

But then for $\sigma>\frac{1}{r}$,
\[
 \|q_{12}(u_+, v_+)\|_{\^{L}^r_{x, t}}\lesssim \|\Lambda^\sigma u_0\|_{\^{L}_x^r}\|v_0\|_{\^{L}_x^r}.
\]
We can trivially exchange $u_0$ and $v_0$ in the above estimate. On the other hand, the convolution constraint $\xi=\eta+(\xi-\eta)$ implies $\<\xi\>^\sigma\lesssim \<\eta\>^\sigma+\<\xi-\eta\>^\sigma$, which gives the following.
\begin{lemma}\label{el_lemma}
 Let $1< r\leq 2$ and $\sigma>\frac{1}{r}$, then
\begin{equation*}
 \|q_{12}(u_+, v_+)\|_{X_{\sigma, 0}^r}\lesssim \|u_0\|_{\^{H}_\sigma^r}\|v_0\|_{\^{H}_\sigma^r}.
\end{equation*}
\end{lemma}

%%%%%%%%%%%%%%%%%%%%%%%%%%%%%%%%%%%%%%%%%%%%%%%%%%%%%%%%%%%%%%%%%%%%%%%%%%%%%%%%%%%%%%%%%%%%%%%%%%%%%%%%%%%%%%%%%%
%%%%%%%%%%%%%%%%%%%%%%%%%%%%%%%%%%%%%%%%%%%%%%%%%%%%%%%%%%%%%%%%%%%%%%%%%%%%%%%%%%%%%%%%%%%%%%%%%%%%%%%%%%%%%%%%%%
%%%%%%%%%%%%%%%%%%%%%%%%%%%%%%%%%%%%%%%%%%%%%%%%%%%%%%%%%%%%%%%%%%%%%%%%%%%%%%%%%%%%%%%%%%%%%%%%%%%%%%%%%%%%%%%%%%

\subsection{The hyperbolic case}
We again choose $s_{1,2}\in (0, \frac{1}{r})$ with $s_1+s_2=\frac{1}{r}$. As in the elliptic case, we will be only interested in the cases $s_1=0, s_2=\frac{1}{r}$ and $s_2=0, s_1=\frac{1}{r}$.

We further split the hyperbolic case into the \textit{low frequency} and \textit{high frequency} cases respectively
\begin{align}
  |\eta|+|\xi-\eta| & \leq 2|\xi|, \label{hyp_low}\\
  |\eta|+|\xi-\eta| & > 2|\xi|. \label{hyp_high}
\end{align}

The low frequency case \eqref{hyp_low} is similar to the elliptic case. Indeed, using \eqref{q_bound_hyp} and H\"{o}lder's estimate as before, we will have
\begin{equation*}
 |\FT q_{12}(u_+, v_+)| \lesssim I^{\frac{1}{r}} \times \left (\int_{P_-(\eta)=\tau} \frac{dS_\eta}{|\nabla P_- (\eta)|} \left |\^{\Lambda^{s_1} u_0}(\eta)\^{\Lambda^{s_2} v_0}(\xi-\eta)\right |^{r'}\right )^\frac{1}{r'},
\end{equation*}
where
\begin{align*}
 I & =\int_{P_-(\eta)=\tau} \frac{dS_\eta}{|\nabla P_\pm (\eta)|} \left |\frac{\eta\wedge(\xi-\eta)}{|\eta||\xi-\eta|}\right|^r |\eta|^{-s_1 r}|\xi-\eta|^{-s_2 r}\\
 &\qquad\lesssim ||\xi|-|\tau||^\frac{r}{2}|\xi|^\frac{r}{2} \int\limits_{|\eta|+|\xi-\eta|\leq 2|\xi|} \frac{\delta(\tau-|\eta|+|\xi-\eta|)}{|\eta|^{(s_1+\frac{1}{2})r}|\xi-\eta|^{(s_2+\frac{1}{2})r}}\, d\eta.
\end{align*}
Now using \cite[Proposition 4.5]{foschi-klainerman:bilinear}, we obtain for the case $s_1=0, s_2=\frac{1}{r}$ (the other case is similar)
\begin{equation*}
 I\lesssim ||\xi|-|\tau||^{\frac{r}{2}}|\xi|^\frac{r}{2} |\xi|^A (||\xi|-|\tau||)^B,
\end{equation*}
where in the region $0\leq \tau \leq |\xi|$
\begin{align*}
 \ts A & \ts =\max \{(s_2+\frac{1}{2})r, \frac{3}{2}\}-(s_1+\frac{1}{2})r-(s_2+\frac{1}{2})r=-\frac{r}{2},\\
 \ts B & \ts =1-\max \{(s_2+\frac{1}{2})r, \frac{3}{2}\}=-\frac{r}{2}.
\end{align*}
and in the region $-|\xi|\leq \tau\leq 0$,
\begin{align*}
 \ts A & \ts =\max \{(s_1+\frac{1}{2})r, \frac{3}{2}\}-(s_1+\frac{1}{2})r-(s_2+\frac{1}{2})r=\frac{1}{2}-r,\\
 \ts B & \ts =1-\max \{(s_1+\frac{1}{2})r, \frac{3}{2}\}=-\frac{1}{2}.
\end{align*}
But then in the region $0\leq \tau \leq |\xi|$,
\begin{equation*}
 I\lesssim \frac{||\xi|-|\tau||^{\frac{r}{2}}|\xi|^\frac{r}{2}}{||\xi|-|\tau||^{\frac{r}{2}}|\xi|^{\frac{r}{2}}}\lesssim 1,
\end{equation*}
while in the region $-|\xi|\leq \tau\leq 0$,
\begin{equation*}
 I\lesssim \frac{||\xi|-|\tau||^{\frac{r}{2}}|\xi|^\frac{r}{2}}{||\xi|-|\tau||^{\frac{1}{2}}|\xi|^{r-\frac{1}{2}}}\lesssim 1.
\end{equation*}

In the high frequency case \eqref{hyp_high}, we will instead use \cite[Lemma 4.4]{foschi-klainerman:bilinear}\footnote{An examination of the proof of \cite[Lemma 4.4]{foschi-klainerman:bilinear} shows that the lower bound of the resulting integral in the case $|\eta|+|\xi-\eta|\geq c|\xi|$ is exactly $c$}, which gives
\begin{align*}
 I &=\frac{||\xi|-|\tau||^\frac{r}{2}|\xi|^\frac{r}{2}}{(\tau^2-|\xi|^2)^\frac{1}{2}}\int_{2}^\infty |x|\xi|+\tau|^{-(s_1+\frac{1}{2})r+1} |x|\xi|-\tau|^{-(s_2+\frac{1}{2})r+1}(x^2-1)^{-\frac{1}{2}}\, dx\\
 &\qquad\lesssim \frac{||\xi|-|\tau||^{\frac{r}{2}-\frac{1}{2}}|\xi|^\frac{r}{2}}{(|\xi|+|\tau|)^\frac{1}{2}} |\xi|^{1-r} \int_{2}^\infty \ts|x+\frac{\tau}{|\xi|}|^{-(s_1+\frac{1}{2})r+1}|x-\frac{\tau}{|\xi|}|^{-(s_2+\frac{1}{2})r+1}(x^2-1)^{-\frac{1}{2}}\, dx\\
 & \qquad\qquad\lesssim \int_2^\infty x^{2-(1+r)-1}\, dx=\int_2^\infty x^{-r}\, dx\lesssim 1.
\end{align*}

Proceeding similar to the elliptic case, we will obtain the following.
\begin{lemma}\label{hyp_lemma}
  Let $1\leq r\leq 2$ and $\sigma>\frac{1}{r}-\frac{1}{2}$, then
\begin{equation*}
 \|q_{12}(u_+, v_-)\|_{X_{\sigma, 0}^r}\lesssim \|u_0\|_{\^{H}_\sigma^r}\|v_0\|_{\^{H}_\sigma^r}.
\end{equation*}
\end{lemma}

The main estimate \eqref{bi_q_free} now follows from Lemmas \ref{el_lemma} and \ref{hyp_lemma}.

\begin{remark}
 We observe that if instead of $Q_{12}$ one has $Q_{0j}$ for $j=1, 2,$ in \eqref{Q}, then instead of  bounds \eqref{q_bound_el}-\eqref{q_bound_hyp}, one will have the following bounds for the normalized multiplier of $Q_{0j}$ (see \cite[Lemma 13.2]{foschi-klainerman:bilinear} for details)
\begin{align*}
 \frac{||\eta|(\xi-\eta)-|\xi-\eta|\eta|}{|\eta||\xi-\eta|} & \approx \frac{(|\eta|+|\xi-\eta|)^{\frac{1}{2}}(|\eta|+|\xi-\eta|-\xi|)^{\frac{1}{2}}}{|\eta|^\frac{1}{2}|\xi-\eta|^{\frac{1}{2}}},\\
 \frac{||\eta|(\xi-\eta)-|\xi-\eta|\eta|}{|\eta||\xi-\eta|} & \approx \frac{|\xi|^{\frac{1}{2}}(|\xi|-||\eta|-|\xi-\eta||)^{\frac{1}{2}}}{|\eta|^\frac{1}{2}|\xi-\eta|^{\frac{1}{2}}}.
\end{align*}
The estimates for the $Q_{0j}$ will then differ from those for $Q_{12}$ only in having a factor of $\tau^\frac{r}{2}$ instead of $|\xi|^\frac{r}{2}$ in the elliptic case, which can be taken care of just as before. Thus the result of the main theorem for the model problem trivially extends to the $Q_{0j}$ nonlinearity as well.
\end{remark}

\begin{remark}
 For $Q_0$, one has the following bounds for the normalized multiplier (see \cite[Lemma 13.2]{foschi-klainerman:bilinear} for details)
\begin{align*}
 \frac{|\eta||(\xi-\eta)|-\eta\cdot (\xi-\eta)}{|\eta||\xi-\eta|} & \approx \frac{(|\eta|+|\xi-\eta|)(|\eta|+|\xi-\eta|-\xi|)}{|\eta||\xi-\eta|},\\
 \frac{|\eta||(\xi-\eta)|+\eta\cdot (\xi-\eta)}{|\eta||\xi-\eta|} & \approx \frac{|\xi|(|\xi|-||\eta|-|\xi-\eta||)}{|\eta||\xi-\eta|}.
\end{align*}
Using these bounds, the bilinear estimate for $Q_0$ will follow from the uniform bound on
\[
I=|\tau-|\xi||^r\tau^r\int \frac{\delta(\tau-|\eta|-|\xi-\eta|)}{|\eta|^{(s_1+1)r}|\xi-\eta|^{(s_2+1)r}}\, d\eta
\]
\end{remark}
in the elliptic case, and
\[
I=|\tau-|\xi||^r|\xi|^r\int \frac{\delta(\tau-|\eta|+|\xi-\eta|)}{|\eta|^{(s_1+1)r}|\xi-\eta|^{(s_2+1)r}}\, d\eta
\]
in the hyperbolic case. These uniform bounds can be shown in exactly the same way as before. Thus the theorem for the model problem will extend to include the $Q_0$ null-form as well.

%%%%%%%%%%%%%%%%%%%%%%%%%%%%%%%%%%%%%%%%%%%%%%%%%%%%%%%%%%%%%%%%%%%%%%%%%%%%%%%%%%%%%%%%%%%%%%%%%%%%%%%%%%%%%%%%%%
%%%%%%%%%%%%%%%%%%%%%%%%%%%%%%%%%%%%%%%%%%%%%%%%%%%%%%%%%%%%%%%%%%%%%%%%%%%%%%%%%%%%%%%%%%%%%%%%%%%%%%%%%%%%%%%%%%
%%%%%%%%%%%%%%%%%%%%%%%%%%%%%%%%%%%%%%%%%%%%%%%%%%%%%%%%%%%%%%%%%%%%%%%%%%%%%%%%%%%%%%%%%%%%%%%%%%%%%%%%%%%%%%%%%%

\section{The Ward wave map problem: trilinear estimates}\label{sec4}
In order to obtain the local well-posedness for the Ward wave map Cauchy problem, we need to establish the trilinear analog of the bilinear estimate \eqref{null} to take care of the cubic nonlinearity in \eqref{ward_null}. Thus, we need the trilinear estimate
\[
 \|wQ(u, v)\|_{X_{s-1, 0}^r}\lesssim \|w\|_{X_{s, b}^r}\|u\|_{X_{s, b}^r}\|v\|_{X_{s, b}^r},
\]
where $Q(u, v)$ is any of the three basic null-forms. But the last estimate will follow from \eqref{null} and the following multiplicative estimate
\[
 \|wQ(u, v)\|_{X_{s-1, 0}^r}\lesssim \|w\|_{X_{s, b}^r}\|Q(u, v)\|_{X_{s-1, 0}^r}.
\]
Thus, it suffices to prove
\begin{equation}\label{mult}
 X_{s, b}^r\cdot X_{s-1, 0}^r\embed X_{s-1, 0}^r,
\end{equation}
provided $1<r\leq 2$, $s>\frac{n}{r}$, $b>\frac{1}{r}$. Using the triangle inequality on the frequency side, it is not hard to see the following estimate,
\begin{equation*}
 \Lambda^\alpha(\phi\psi)\precsim(\Lambda^\alpha \phi)\psi+\phi\Lambda^\alpha \psi, \qquad \forall \alpha>0
\end{equation*}
for all $\phi$ and $\psi$ with $\~{\phi}, \~{\psi}\geq 0$, where $u\precsim v$ denotes the pointwise estimate $\~{u}(\tau, \xi)\leq \~{v}(\tau, \xi)$. Then, the proof of \eqref{mult} reduces to the following two embeddings:
\begin{align}
 X_{s, b}^r\cdot \^{L}^r & \embed \^{L}^r, \label{mult1}\\
 X_{1, b}^r\cdot X_{s-1, 0}^r & \embed \^{L}^r \label{mult2}.
\end{align}
The first of these, \eqref{mult1}, follows trivially from Young's inequality and appropriate H\"{o}lder's inequalities to show that
\[
 \^{L}^\infty_{t, x} \embed X_{s, b}^r.
\]
The second embedding \eqref{mult2}, is equivalent to the estimate
\[
 \|fg\|_{\^{L}^r_{t, x}}\lesssim \|f\|_{X_{1, b}^r}\|g\|_{X_{s-1, 0}^r},
\]
which we now prove. Using the definition of $\^{L}^r$, we have by Young's inequality
\begin{align*}
 \|fg\|_{\^{L}^r}=\|\~{f}*_{\tau, \xi}\~{g}\|_{L^{r'}}\lesssim \left \|\|\~{f}\|_{L^1_\tau}*_\xi \|\~{g}\|_{L^{r'}_\tau}\right \|_{L^{r'}_\xi}\lesssim \left \|\|\~{f}\|_{L^1_\tau}\right \|_{L^p_\xi} \left \|\|\~{g}\|_{L^{r'}_\tau}\right \|_{L^q_\xi},
\end{align*}
for some $p, q$, satisfying
\begin{equation}\label{exps_young}
 \frac{1}{r'}+1=\frac{1}{p}+\frac{1}{q}.
\end{equation}
Using H\"{o}lder's inequality, we can bound the previous by
\[
 \lesssim\left \|\<\xi\>\|\~{f}\|_{L^1_\tau}\right \|_{L^{r'}_\xi}\|\<\xi\>^{-1}\|_{L^m}\left \|\<\xi\>^{s-1}\|\~{g}\|_{L^{r'}_\tau}\right \|_{L^{r'}_\xi}\|\<\xi\>^{-(s-1)}\|_{L^l},
\]
where
\begin{equation}\label{exps_holder}\begin{split}
 \frac{1}{p} & =\frac{1}{r'}+\frac{1}{m},\\
 \frac{1}{q} & =\frac{1}{r'}+\frac{1}{l}.
\end{split}
\end{equation}
Notice that, since $b>\frac{1}{r}$,
\begin{align*}
 \left \|\<\xi\>\|\~{f}\|_{L^1_\tau}\right \|_{L^{r'}_\xi} & \lesssim \left \|\<\xi\>\|\<|\tau|-|\xi|\>^b\~{f}\|_{L^{r'}_\tau}\|\<|\tau|-|\xi|\>^{-b}\|_{L^r_\tau}\right \|_{L^{r'}_\xi}\\
 &\qquad  \lesssim \left \|\<\xi\>\|\<|\tau|-|\xi|\>^b\~{f}\|_{L^{r'}_\tau}\right \|_{L^{r'}_\xi}=\|f\|_{X_{1, b}^r},
\end{align*}
and
\[
 \left \|\<\xi\>^{s-1}\|\~{g}\|_{L^{r'}_\tau}\right \|_{L^{r'}_\xi}=\|g\|_{X_{s-1, 0}^r}.
\]
But then it only remains to show that there is an appropriate choice of $m$ and $l$ in \eqref{exps_holder}, which insures that
\begin{equation}\label{needed_bound}
 \|\<\xi\>^{-1}\|_{L^m}, \|\<\xi\>^{-(s-1)}\|_{L^l}\lesssim 1.
\end{equation}
For this we need $m>2$ and $l(s-1)>2$. The later will follow from $l>2r$, since $s-1>\frac{1}{r}$. So let us choose
\[
 l=2r+\epsilon.
\]
From \eqref{exps_holder}-\eqref{exps_young} we will then have
\[
 \frac{1}{r'}+1=\frac{1}{r'}+\frac{1}{m}+\frac{1}{r'}+\frac{1}{l},
\]
from which we can solve for
\[
 \frac{1}{m}=1-\frac{1}{r'}-\frac{1}{l}=\frac{1}{r}-\frac{1}{2r+\epsilon}.
\]
If $r=2$, then obviously $\frac{1}{m}<\frac{1}{2}$, which would imply that $m>2$. On the other hand, if $r=1+\delta$ for some $0<\delta<1$, then
\[
 \frac{1}{m}=\frac{1}{1+\delta}-\frac{1}{2(1+\delta)+\epsilon}<\frac{1}{2},
\]
provided
\[
 2(1+\delta)+\epsilon<\frac{1}{\frac{1}{1+\delta}-\frac{1}{2}}=\frac{2(1+\delta)}{1-\delta},
\]
or
\[
 0<\epsilon<\frac{2(1+\delta)}{1-\delta}-2(1+\delta).
\]
Since the right hand side of the last inequality is positive, due to $0<\delta<1$, there is an $\epsilon$, for which $l>2r$ and $m>2$, thus guaranteeing \eqref{needed_bound} and finishing the proof of \eqref{mult2}.

%%%%%%%%%%%%%%%%%%%%%%%%%%%%%%%%%%%%%%%%%%%%%%%%%%%%%%%%%%%%%%%%%%%%%%%%%%%%%%%%%%%%%%%%%%%%%%%%%%%%%%%%%%%%%%%%%%
%%%%%%%%%%%%%%%%%%%%%%%%%%%%%%%%%%%%%%%%%%%%%%%%%%%%%%%%%%%%%%%%%%%%%%%%%%%%%%%%%%%%%%%%%%%%%%%%%%%%%%%%%%%%%%%%%%
%%%%%%%%%%%%%%%%%%%%%%%%%%%%%%%%%%%%%%%%%%%%%%%%%%%%%%%%%%%%%%%%%%%%%%%%%%%%%%%%%%%%%%%%%%%%%%%%%%%%%%%%%%%%%%%%%%

\appendix

%%%%%%%%%%%%%%%%%%%%%%%%%%%%%%%%%%%%%%%%%%%%%%%%%%%%%%%%%%%%%%%%%%%%%%%%%%%%%%%%%%%%%%%%%%%%%%%%%%%%%%%%%%%%%%%%%%
%%%%%%%%%%%%%%%%%%%%%%%%%%%%%%%%%%%%%%%%%%%%%%%%%%%%%%%%%%%%%%%%%%%%%%%%%%%%%%%%%%%%%%%%%%%%%%%%%%%%%%%%%%%%%%%%%%
%%%%%%%%%%%%%%%%%%%%%%%%%%%%%%%%%%%%%%%%%%%%%%%%%%%%%%%%%%%%%%%%%%%%%%%%%%%%%%%%%%%%%%%%%%%%%%%%%%%%%%%%%%%%%%%%%%

\section{The transfer principle}\label{appA}
In this appendix we state and prove the transfer principle for elements of $X_{s, b}^r$ spaces. This is very similar to the $L^2$ case, although, due to the nature of our spaces, we stick to the frequency side, when proving the transfer principle. We closely follow \cite[Section 3.2]{selberg_thesis} for the first half of the appendix. See also \cite[Propositions 3.6, 3.7]{kl-sel:surrvey} and \cite[Lemma 2.1]{grunrock:mkdv}.

Given a function $u\in X_{s, b}^r$, it can be uniquely decomposed as
\[
 u=u_++u_-,
\]
where $\~{u_+}$ is supported in $[0, \infty)\times \R^n$ and $\~{u_-}$ is supported in $(-\infty, 0]\times \R^n$. Clearly $u_\pm\in X_{s, b}^r$, and
\[
 \|u\|_{X_{s, b}^r}^{r'}=\|u_+\|_{X_{s, b}^r}^{r'}+\|u_-\|_{X_{s, b}^r}^{r'}.
\]
When $b>\frac{1}{r}$, one has the following characterization of $X_{s, b}^r$.

\begin{prop}\label{int_rep_prop}
 If $b>\frac{1}{r}$, then
\begin{enumerate}
 \item [(a)] $X_{s, b}^r\subset C_b(\R, \^{H}_s^r)$ in the sense that any tempered distribution $u\in X_{s, b}^r$ has a unique representative 
\[
 t\mapsto u(t) \qquad \text{in}\quad C_b(\R, \^{H}_s^r),
\]
and
\[
 \|u(t)\|_{\^{H}_s^r}\leq C\|u\|_{X_{s, b}^r} \qquad \text{for all } t\in\R,
\]
where $C$ depends only on $b$ and $r$.

 \item [(b)] $u\in X_{s, b}^r$ iff there exist $f_+, f_-\in L^{r'}(\R, \^{H}_s^r)$, such that
\begin{align*}
& \^{f_+}(\rho)(\xi)=0 \qquad \text{for } |\xi|<-\rho,\\
& \^{f_-}(\rho)(\xi)=0 \qquad \text{for } |\xi|<\rho,
\end{align*}
and
\begin{equation}\label{int_rep}
 u_\pm(t)=\frac{1}{2\pi}\int_{-\infty}^\infty \frac{e^{it(\rho\pm D)}f_\pm(\rho)}{(1+|\rho|)^b}\, d\rho.
\end{equation}
Moreover,
\[
 \|u_\pm\|_{X_{s, b}^r}=\|f_\pm\|_{L^{r'}(\R, \^{H}_s^r)}.
\]
\end{enumerate}
\end{prop}

\begin{remark}
 The representation \eqref{int_rep} is equivalent to a decomposition with respect to the foliation of the Fourier space by the two families of cones
\[
 \sN_+(\rho): \tau=|\xi|+\rho, |\xi|>-\rho, \qquad \text{and } \qquad \sN_-(\rho): \tau=-|\xi|+\rho, |\xi|>\rho,
\]
where the family parameter $\rho\in(-\infty, \infty)$. Indeed, we can write
\begin{equation}\label{int_rep_F}\begin{split}
 \~{u_\pm}(\tau, \xi)
 &=\int \~{u_\pm}(\rho\pm |\xi|, \xi)\delta(\rho-(\tau\mp |\xi|))\, d\rho\\
 &=\int \frac{\^{f_\pm}(\rho)(\xi)}{(1+|\rho|)^b}\delta(\rho-(\tau\mp |\xi|))\, d\rho,
\end{split}\end{equation}
where
\[
 \^{f_\pm}(\rho)(\xi)\delta(\rho-(\tau\mp |\xi|))=\~{u_\pm}(\tau, \xi)(1+|\rho|)^b\delta(\rho-(\tau\mp |\xi|))
\]
define $W^{s, r'}$ measures on the cones $\sN_\pm(\rho)$, and as such, represent translations of the Fourier transforms of $\^{H}_s^r$ solutions of the free wave equation restricted respectively to the upper and lower Fourier half-spaces.

Taking the time inverse Fourier transforms of \eqref{int_rep_F}, we have
\[
 \^{u_\pm}(\xi)=\frac{1}{2\pi}\int \frac{e^{it(\rho\pm|\xi|)}\^{f_\pm}(\rho)(\xi)}{(1+|\rho|)^b}\, d\rho.
\]
Taking the space inverse Fourier transform of this identity will result in \eqref{int_rep}.
\end{remark}

\begin{proof}
 The existence of part (a) follows from part (b). Furthermore, from \eqref{int_rep}, H\"{o}lder's inequality and the fact that $b>\frac{1}{r}$,
\begin{align*}
 \|u(t)\|_{\^{H}_s^r}
 & \leq \int \frac{\|f_+(\rho)\|_{\^{H}_s^r}}{(1+|\rho|)^b}\, d\rho+\int \frac{\|f_-(\rho)\|_{\^{H}_s^r}}{(1+|\rho|)^b}\, d\rho\\
 & \leq C\left (\|f_+\|_{L^{r'}(\R, \^{H}_s^r)}+\|f_-\|_{L^{r'}(\R, \^{H}_s^r)}\right )\\
 & \leq C\|u\|_{X_{s, b}^r},
\end{align*}
where $C=2^{r'}\left (\int (1+|\rho|)^{-rb}\, d\rho\right )^{\frac{1}{r}}$.

The uniqueness of part (a) is straightforward, so it remains to prove part (b). 

Let us define the multipliers
\begin{align*}
 \^{\Lambda u} & = \<\xi\> \^{u}\\
 \~{\Lambda_-u} & = \<|\tau|-|\xi|\>\~{u}.
\end{align*}
Using these, we define the isometry
\[
 u\longleftrightarrow F=\FT(\Lambda^s\Lambda_-^b u), \qquad X_{s, b}^r\longleftrightarrow L^{r'},
\]
under which $u_+$ and $u_-$ correspond to
\[
 F_+=\chi_{[0, \infty)\times \R^n}F, \qquad \text{and} \qquad F_-=\chi_{(-\infty, 0]\times \R^n}F
\]
respectively. We define another isometry,
\[
 F_\pm \longleftrightarrow f_\pm, \qquad L^{r'}(\R^{1+n})\longleftrightarrow L^{r'}(\R, \^{H}_s^r)
\]
by
\begin{align*}
 (1+|\xi|)^s\^{f_+(\rho)}(\xi) & = F_+(\rho+|\xi|, \xi),\\
 (1+|\xi|)^s\^{f_-(\rho)}(\xi) & = F_-(\rho-|\xi|, \xi).
\end{align*}
Notice that under the composition of the isometries,
\[
 \^{f_\pm}(\rho)(\xi)=\frac{1}{(1+|\xi|)^s}F_\pm (\rho\pm |\xi|, \xi)=(1+|\rho|)^b\~{u_\pm}(\rho\pm |\xi|, \xi).
\]
It is easy to check that $f_\pm$ is in $L^{r'}(\R, \^{H}_s^r)$, iff $F_\pm$ is in $L^{r'}(\R^{1+n})$. Thus, $u \in X_{s, b}^r$, iff $f_+, f_- \in L^{r'}(\R, \^{H}_s^r)$, and we set
\[
 v_+(t)=\frac{1}{2\pi}\int_{-\infty}^\infty \frac{e^{it(\rho+ D)}f_+(\rho)}{(1+|\rho|)^b}\, d\rho, \qquad \text{and} \qquad v_-(t)=\frac{1}{2\pi}\int_{-\infty}^\infty \frac{e^{it(\rho-D)}f_-(\rho)}{(1+|\rho|)^b}\, d\rho
\]
for $t\in\R$. By the dominated convergence theorem $v_+, v_-\in C(\R, \^{H}_s^r)$, and $v_\pm$ are tempered distributions, as the $\^{H}_s^r$ norm of $v_\pm(t)$ is bounded uniformly in $t$. We next prove that $u_+=v_+$ in the sense of distributions. The proof of $u_-=v_-$ is similar.

We have for $\phi\in \sS(\R^{1+n})$,
\begin{align*}
 \<v_+, \phi\>
 & =\int \<v_+(t), \phi(t)\>\, dt\\
 & =\int \left \<\frac{1}{2\pi}\int \frac{e^{it(\rho+ D)}f_+(\rho)}{(1+|\rho|)^b}\, d\rho, \phi(t)\right \>\, dt\\
 & =\iint \left \<\frac{1}{2\pi} \frac{e^{it(\rho+ D)}f_+(\rho)}{(1+|\rho|)^b}, \phi(t)\right \>\, d\rho dt\\
 & =\iiint \frac{1}{2\pi} \frac{e^{it(\rho+ |\xi|)}\^{f_+}(\rho)(\xi)}{(1+|\rho|)^b} \FT_\xi^{-1}\phi(t)(\xi)\, d\xi d\rho dt\\
 & =\iiint \frac{1}{2\pi} e^{it(\rho+ |\xi|)}\~{u_+}(\rho+|\xi|, \xi) \FT_\xi^{-1}\phi(t)(\xi)\, d\xi d\rho dt\\
 & =\iint \~{u_+}(\tau, \xi)\left (\frac{1}{2\pi}\int e^{it\tau} \FT_\xi^{-1} \phi(t)(\xi)\, dt\right )\, d\tau d\xi\\
 & =\iint \~{u_+}(\tau, \xi)\FT_{\tau, \xi}^{-1}\phi (\tau, \xi)\, d\tau d\xi\\
 & =\<u_+, \phi\>.
\end{align*}
\end{proof}

Using the integral representation \eqref{int_rep}, we will prove a useful corollary, which allows one to transfer multilinear estimates involving solutions of the free wave equation to corresponding estimates for elements of $X_{s, b}^r$ spaces with $b>\frac{1}{r}$. This result is appropriately called the {\it transfer principle}.

\begin{prop}\label{transfer_prop}
 Assume that $T:\^{H}_{s_1}^r(\R^n)\times \dots \times \^{H}_{s_k}^r(\R^n)\to \^{H}_\sigma^r(\R^n)$ is a $k$-linear operator, and let $b>\frac{1}{r}$.
\begin{enumerate}
 \item [(a)] If
\begin{equation}\label{freewave_mult}
 \|T(e^{\lambda_1 itD}f_1, \dots, e^{\lambda_k itD}f_k)\|_{\^{L}_t^{p}(\^{L}_x^q)}\leq C\|f_1\|_{\^{H}_{s_1}^r}\dots \|f_k\|_{\^{H}_{s_k}^r},
\end{equation}
where $\lambda$ is a fixed $k$-tuple in $\{-1, 1\}$, then
\begin{equation}\label{u_mult}
 \|T(u_1, \dots, u_k)\|_{\^{L}_t^{p}(\^{L}_x^q)}\leq C\|u_1\|_{X_{s_1, b}^r}\dots \|u_k\|_{X_{s_k, b}^r}
\end{equation}
for all $(u_1, \dots, u_k)\in X_{s_1, b}^r\times \dots \times X_{s_k, b}^r$, such that
\begin{equation}\label{signs}
 \supp \^{u_j}\subseteq \left \{
\begin{array}{ll}
 [0, \infty)\times \R^n   &\qquad \text{if } \lambda_j=1\\
 (-\infty, 0]\times \R^n &\qquad \text{if } \lambda_j=-1.
\end{array}\right .
\end{equation}

 \item [(b)] If \eqref{freewave_mult} holds for all $\lambda\in\{-1, 1\}^k$, then \eqref{u_mult} holds for all $(u_1, \dots, u_k)\in X_{s_1, b}^r\times \dots \times X_{s_k, b}^r$.
\end{enumerate}
\end{prop}

\begin{proof}
We assume that $T$ is bilinear to keep the notation simple, the general case is similar. Let us denote $U_\pm=e^{\pm itD}$. From \eqref{int_rep},
\[
 \~{u_j}(\tau, \xi)=\frac{1}{2\pi}\int \frac{\FT_{t, x} (e^{it\rho}U_jf_j(\rho))}{(1+|\rho|)^b}\, d\rho=\frac{1}{2\pi}\int \frac{\~{U_jf_j(\rho)}(\tau-\rho, \xi)}{(1+|\rho|)^b}\, d\rho,
\]
where $U_j=U_+$, or $U_j=U_-$, depending on whether $\lambda_j=1$, or $\lambda_j=-1$. Denoting the multiplier of $T$ by $m$, and using $\Xi=(\tau, \xi)$, we have
\begin{align*}
 & \FT T(u_1, u_2)(\Xi)\\
 &\qquad\qquad =\iint m(\Xi_1, \Xi_2)\~{u_1}(\Xi_1)\~{u_2}(\Xi_2)\delta (\Xi-\Xi_1-\Xi_2)\, d\Xi_1 d\Xi_2\\
 &\qquad\qquad =\iint m(\Xi_1, \Xi_2) \frac{1}{2\pi}\int \frac{\~{U_1f_1(\rho_1)}(\tau_1-\rho_1, \xi_1)}{(1+|\rho_1|)^b}\, d\rho_1\\
 &\qquad\qquad\qquad\qquad\qquad\qquad\qquad \times \frac{1}{2\pi}\int \frac{\~{U_2f_2(\rho_2)}(\tau_2-\rho_2, \xi_2)}{(1+|\rho_2|)^b}\, d\rho_2\, \delta (\Xi-\Xi_1-\Xi_2)\, d\Xi_1 d\Xi_2\\
 &\qquad\qquad =\frac{1}{4\pi^2}\iint \frac{\FT T(U_1f_1(\rho_1), U_2f_2(\rho_2))(\tau+\rho_1+\rho_2, \xi)}{(1+|\rho_1|)^b(1+|\rho_2|)^b}\, d\rho_1d\rho_2,
\end{align*}
where we used Fubini's theorem in the last step. Then using Minkowski's inequality, \eqref{freewave_mult}, H\"{o}lder's inequality and Proposition \ref{int_rep_prop},
\begin{align*}
 \|T(u_1, u_2)\|_{\^{L}_t^{p}(\^{L}_x^q)}
 & \lesssim \iint \frac{\|T(U_1f_1(\rho_1), U_2f_2(\rho_2))\|_{\^{L}_t^{p}(\^{L}_x^q)}}{(1+|\rho_1|)^b(1+|\rho_2|)^b}\, d\rho_1d\rho_2\\
 & \lesssim \iint \frac{\|f_1(\rho_1)\|_{\^{H}_{s_1}^r}\|f_2(\rho_2)\|_{\^{H}_{s_k}^r}}{(1+|\rho_1|)^b(1+|\rho_2|)^b}\, d\rho_1d\rho_2\\
 & \lesssim \|f_1\|_{L^{r'}(\^{H}_{s_1}^r)}\|f_2\|_{L^{r'}(\^{H}_{s_2}^r)}\\
 & = \|u_1\|_{X_{s_1, b}^r}\|u_2\|_{X_{s_2, b}^r}.
\end{align*}
\end{proof}

%%%%%%%%%%%%%%%%%%%%%%%%%%%%%%%%%%%%%%%%%%%%%%%%%%%%%%%%%%%%%%%%%%%%%%%%%%%%%%%%%%%%%%%%%%%%%%%%%%%%%%%%%%%%%%%%%%
%%%%%%%%%%%%%%%%%%%%%%%%%%%%%%%%%%%%%%%%%%%%%%%%%%%%%%%%%%%%%%%%%%%%%%%%%%%%%%%%%%%%%%%%%%%%%%%%%%%%%%%%%%%%%%%%%%
%%%%%%%%%%%%%%%%%%%%%%%%%%%%%%%%%%%%%%%%%%%%%%%%%%%%%%%%%%%%%%%%%%%%%%%%%%%%%%%%%%%%%%%%%%%%%%%%%%%%%%%%%%%%%%%%%%

\section{The general well-posedness theorem}\label{appB}
In this appendix we will state the general well-posedness theorem for nonlinear wave equations with data in $\^{H}_s^r\times \^{H}_{s-1}^r$. We do not include the proof here, as, with minor differences, it is similar to the proof of the analogous theorem in the $L^2$ case, for which we refer the reader to \cite[Theorem 4.1]{selberg_thesis} (see also \cite[Theorem 5.3]{kl-sel:surrvey}).

The proof of the general well-posedness theorem follows from appropriate estimates for the solution to the linear wave equation in the relevant solution spaces, which we state first. This result is again similar to the analogous result in the $L^2$ case.

Consider the following Cauchy problem for the linear wave equation,
\begin{align}
 & \Box u=F(t, x), \quad (t, x)\in \R^{1+n}, \label{lin_wave}\\
 & (u, \del_t u)_{t=0} = (f, g)\in \^{H}_s^r\times \^{H}_{s-1}^r.\label{lin_wave_IC}
\end{align}

\begin{theorem}[c.f. Theorem 13 in \cite{selberg_thesis}]\label{LW}
 Assume $s\in \R$, $\frac{1}{r}<b<1$, $\epsilon \in (0, 1-b)$, $F\in X_{s-1, b+\epsilon-1}^r$ and
\[
 \chi \in C_c^\infty (\R), \qquad \chi=1 \text{ on } [-1, 1], \qquad \supp \chi\subseteq (-2, 2).
\]
Let $0<T<1$ and define
\[
 u(t)=\chi(t) u_0+\chi(t/T)u_1+u_2,
\]
where
\begin{align*}
 u_0 & = \cos (tD)\cdot f+D^{-1} \sin (tD) g,\\
 u_1 & =\int_0^tD^{-1} \sin ((t-t')D)\cdot F_1(t')\, dt',\\
 u_2 & =\Box^{-1} F_2,
\end{align*}
and
\[
 F=F_1+F_2=\phi(T^{1/2}\Lambda_-)F+(1-\phi(T^{1/2}\Lambda_-))F,
\]
with
\[
 \phi\in C_c^\infty(\R), \qquad \phi=1 \text{ on } [-2, 2], \qquad \supp\phi\subseteq (-4, 4).
\]
Then
\[
  \|u\|_{Z_{s, b}^r}\leq C\left (\|f\|_{\^{H}_s^r}+\|g\|_{\^{H}_{s-1}^r}+T^{\epsilon/2}\|F\|_{X_{s-1, b+\epsilon-1}^r}\right ),
\]
where $C$ depends only on $\chi$ and $b$. Furthermore, $u$ is the unique solution of the Cauchy problem \eqref{lin_wave}-\eqref{lin_wave_IC} such that $u\in C(0, T], \^{H}_s^r)\cap C^1[(0, T], \^{H}_{s-1}^r)$.
\end{theorem}

Using this theorem, one can tackle the local well-posedness of a nonlinear Cauchy problem via an iteration argument.

Consider the Cauchy problem
\begin{align}
 & \Box u=F(u, \del u), \quad (t, x)\in \R^{1+n}, \label{NLW_eq}\\
 & (u, \del_t u)|_{t=0}=(f, g)\in \^{H}_s^r\times \^{H}_{s-1}^r \label{NLW_ic},
\end{align}
where $\del u$ is the space-time gradient of $u$, and $F$ is a smooth function with $F(0)=0$.

\begin{theorem}[c.f Theorem 14 in \cite{selberg_thesis}]\label{gen_LWP}
 Assume $s\in\R$, $\frac{1}{r}<b<1$, $\epsilon\in (0, 1-b)$. If
\begin{equation*}\label{F_b}
 \|F(u, \del u)\|_{X_{\sigma-1, b+\epsilon-1}^r}\leq A_\sigma (\|u\|_{Z_{s, b}^r})\|u\|_{Z_{\sigma, b}^r}
\end{equation*}
for all $\sigma\geq s$, and
\begin{equation*}\label{F_c}
 \|F(u, \del u)-F(v, \del v)\|_{X_{\sigma-1, b+\epsilon-1}^r}\leq A_s (\|u\|_{Z_{s, b}^r}+\|v\|_{Z_{s, b}^r})\|u-v\|_{Z_{\sigma, b}^r},
\end{equation*}
for all $u, v\in Z_{s, b}^r$, where $A_\sigma:\R_+\to\R_+$ is increasing and locally Lipschitz for every $\sigma\geq s$, then
\begin{itemize}
 \item (existence) There exists $u\in Z_{s, b}^r$, which solves \eqref{NLW_eq}-\eqref{NLW_ic} on $[0, T]\times \R^n$, where $T=T(\|f\|_{\^{H}_s^r}+\|g\|_{\^{H}_s^r})>0$ depends continuously on $\|f\|_{\^{H}_s^r}+\|g\|_{\^{H}_s^r}$.
 \item (uniqueness) The solution is uniques, in the class $Z_{s, b}^r$, i.e., if $u, v\in Z_{s, b}^r$ both solve \eqref{NLW_eq}-\eqref{NLW_ic} on $[0, T]\times \R^n$ for some $T>0$, then
\[
 u(t)=v(t) \qquad \text{for } t\in [0, T].
\]
 \item (Lipschitz) The solution map
\[
 (f, g)\mapsto u, \qquad \^{H}_s^r\times \^{H}_{s-1}^r\to Z_{s, b}^r
\]
is locally Lipschitz.
 \item (higher regularity) If the data has higher regularity
\[
 f\in \^{H}_\sigma^r, g\in \^{H}_{\sigma-1}^r \qquad \text{for } \sigma >s,
\]
then $u\in C([0, T], \^{H}_{\sigma}^r)\cap C^1([0, T], \^{H}_{\sigma-1}^r)$ for any $T>0$ for which $u$ solves \eqref{NLW_eq}-\eqref{NLW_ic}. In particular, if $(f, g)\in\sS$, then $u\in C^\infty ([0, T]\times \R^n)$.
\end{itemize}
\end{theorem}

%%%%%%%%%%%%%%%%%%%%%%%%%%%%%%%%%%%%%%%%%%%%%%%%%%%%%%%%%%%%%%%%%%%%%%%%%%%%%%%%%%%%%%%%%%%%%%%%%%%%%%%%%%%%%%%%%%
%%%%%%%%%%%%%%%%%%%%%%%%%%%%%%%%%%%%%%%%%%%%%%%%%%%%%%%%%%%%%%%%%%%%%%%%%%%%%%%%%%%%%%%%%%%%%%%%%%%%%%%%%%%%%%%%%%
%%%%%%%%%%%%%%%%%%%%%%%%%%%%%%%%%%%%%%%%%%%%%%%%%%%%%%%%%%%%%%%%%%%%%%%%%%%%%%%%%%%%%%%%%%%%%%%%%%%%%%%%%%%%%%%%%%

%%%%%%%%%%%%%%%%%%%%%%%%%%%%%%%%%%%%%%%%%%%%%%%%%%%%%%%%%%%%%%%%%%%%%%%%%%%%%%%%%%%%%%%%%%%%%%%%%%%%%%%%%%%%%%%%%%
%%%%%%%%%%%%%%%%%%%%%%%%%%%%%%%%%%%%%%%%%%%%%%%%%%%%%%%%%%%%%%%%%%%%%%%%%%%%%%%%%%%%%%%%%%%%%%%%%%%%%%%%%%%%%%%%%%
%%%%%%%%%%%%%%%%%%%%%%%%%%%%%%%%%%%%%%%%%%%%%%%%%%%%%%%%%%%%%%%%%%%%%%%%%%%%%%%%%%%%%%%%%%%%%%%%%%%%%%%%%%%%%%%%%%


\begin{thebibliography}{10}

\bibitem{DFS_2d}
P. D'Ancona, D. Foschi, S. Selberg \emph{Product estimates for wave-Sobolev spaces in $2+1$ and $1+1$ dimensions}, Nonlinear partial differential equations and hyperbolic wave phenomena, 125–150, Contemp. Math., \textbf{526}, Amer. Math. Soc., Providence, RI, 2010.

\bibitem{christ}
M. Christ, \emph{Power series solution of a nonlinear Schr\"{o}dinger equation}. Mathematical aspects of nonlinear dispersive equations, 131–155, Ann. of Math. Stud., {\bf 163}, Princeton Univ. Press, Princeton, NJ, 2007.

\bibitem{christ:null}
D. Christodoulou, \emph{Global solutions of nonlinear hyperbolic equations for small initial data},  Comm. Pure Appl. Math.  \textbf{39}  (1986),  no. 2, 267--282.

\bibitem{magda:thesis}
M. Czubak, \emph{Well-posedness for the space-time Monopole Equation and Ward Wave Map}, Ph.D. Thesis, University of Texas at Austin, 2008.

\bibitem{dai-terng-uhlenbeck}
B. Dai, Ch. Terng, K. Uhlenbeck, \emph{On the space-time monopole equation},  Surveys in differential geometry. Vol. \textbf{X},  1--30, Surv. Differ. Geom., 10, Int. Press, Somerville, MA, 2006.

\bibitem{grig-tang}
V. Grigoryan, A. Tanguay \emph{Improved well-posedness for the quadratic derivative nonlinear wave equation in 2D}, preprint 2012.

\bibitem{foschi-klainerman:bilinear}
D. Foschi, S. Klainerman, \emph{Bilinear space-time estimates for homogeneous wave equations}. Ann. Sci. \'{E}cole Norm. Sup. (4) \textbf{33} (2000), no. 2, 211–274. 

\bibitem{grunrock:mkdv}
A. Gr\"{u}nrock, \emph{An improved local well-posedness result for the modified KdV equation}. Int. Math. Res. Not. 2004, no. {\bf 61}, 3287–3308.

\bibitem{grunrock:nls}
A. Gr\"{u}nrock, \emph{Bi- and trilinear Schr\"{o}dinger estimates in one space dimension with applications to cubic NLS and DNLS},  Int. Math. Res. Not.  2005,  no. \textbf{41}, 2525--2558.

\bibitem{grunrock-herr}
A. Gr\"unrock, S. Herr, \emph{Low regularity local well-posedness of the derivative nonlinear Schr\"{o}dinger equation with periodic initial data}.  SIAM J. Math. Anal.  \textbf{39}  (2008),  no. 6, 1890--1920.

\bibitem{grunrock:wave}
A. Gr\"{u}nrock, \emph{On the wave equation with quadratic nonlinearities in three space dimensions}. J. Hyperbolic Differ. Equ. {\bf 8} (2011), no. 1, 1–8

\bibitem{hormander}
L. H\"ormander, \emph{The analysis of linear partial differential operators. II. Differential operators with constant coefficients.} Grundlehren der Mathematischen Wissenschaften [Fundamental Principles of Mathematical Sciences], 257. Springer-Verlag, Berlin, 1983. viii+391 pp. ISBN: 3-540-12139-0

\bibitem{klainerman:null1}
S. Klainerman, \emph{Long time behaviour of solutions to nonlinear wave equations},  Proceedings of the International Congress of Mathematicians, Vol. 1, 2 (Warsaw, 1983),  1209--1215, PWN, Warsaw, 1984

\bibitem{klainerman:null2}
S. Klainerman, \emph{The null condition and global existence to nonlinear wave equations}, Lectures in Appl. Math., \textbf{23}, 1986, 293--326.

\bibitem{kl-mac:null}
S. Klainerman, M. Machedon, \emph{Estimates for null forms and the spaces $H_{s,δ}$}. Internat. Math. Res. Notices, 1996, no. \textbf{17}, 853–865.

\bibitem{kl-sel:surrvey}
S. Klainerman, S. Selberg, \emph{Bilinear estimates and applications to nonlinear wave equations}. Commun. Contemp. Math. 4 (2002), no. 2, 223–295.

\bibitem{kl-sel}
S. Klainerman, S. Selberg, \emph{Remark on the optimal regularity for equations of wave maps type}, C.P.D.E., \textbf{22} (1997), 901--918.

\bibitem{lindblad:1}
H. Lindblad, \emph{A sharp counterexample to the local existence of low-regularity solutions to nonlinear wave equations}. Duke Math. J. {\bf 72} (1993), no. 2, 503–539.

\bibitem{lindblad:2}
H. Lindblad, \emph{Counterexamples to local existence for semi-linear wave equations}. Amer. J. Math. {\bf 118} (1996), no. 1, 1–16.

\bibitem{lindblad:3}
H. Lindblad, \emph{Counterexamples to local existence for quasilinear wave equations}. Math. Res. Lett. {\bf 5} (1998), no. 5, 605–622.

\bibitem{ponce-sideris}
G. Ponce, T. Sideris \emph{Local regularity of nonlinear wave equations in three space dimensions},  Comm. Partial Differential Equations  \textbf{18}  (1993),  no. 1-2, 169--177.

\bibitem{selberg_thesis}
S. Selberg \emph{Multilinear space-time estimates and applications to local existence theory for nonlinear wave equations}. Ph.D. Thesis, Princeton University, 1999.

\bibitem{vargas-vega}
A. Vargas, L. Vega \emph{Global wellposedness for 1D non-linear Schr\"{o}dinger equation for data with an infinite $L^2$ norm}, J. Math. Pures Appl. (9) {\bf 80} (2001), no. 10, 1029–1044.

\bibitem{ward:soliton}
R. Ward, \emph{Soliton solutions in an integrable chiral model in $2+1$ dimensions.}  J. Math. Phys.  \textbf{29}  (1988),  no. 2, 386--389.

\bibitem{zhou}
Y. Zhou, \emph{Local existence with minimal regularity for nonlinear wave equations},  Amer. J. Math.  \textbf{119}  (1997),  no. 3, 671--703. 
\end{thebibliography}
\end{document}